\documentclass[11pt]{article}

\usepackage{fullpage}

\usepackage[colorlinks,citecolor=blue,bookmarks=true]{hyperref}

\usepackage{amsmath, amsthm, amssymb}

\makeatletter
\renewenvironment{proof}[1][\proofname]
{\par\pushQED{\qed}
	\normalfont\topsep6\p@\@plus6\p@\relax\trivlist
	\item[\hskip\labelsep\bfseries#1\@addpunct{.}]
	\ignorespaces}
{\popQED\endtrivlist\@endpefalse}
\makeatother

\newtheorem{theo}{Theorem}[section]
\newtheorem{prop}[theo]{Proposition}
\newtheorem{lemma}[theo]{Lemma}
\newtheorem{coro}[theo]{Corollary}

\newtheorem{claim}[theo]{Claim}
\newtheorem{remark}[theo]{Remark}
\newtheorem*{remark*}{Remark}

\newcommand{\FF}{{\cal F}}

\newcommand{\UU}{{\cal U}}
\newcommand{\R}{{\mathbb R}}

\newcommand{\floor}[1]{\left\lfloor{#1}\right\rfloor}
\newcommand{\ceil}[1]{\left\lceil #1 \right\rceil}

\renewcommand{\wp}{\mathrm{span}}
\renewcommand{\a}{\alpha}
\renewcommand{\d}{\delta}
\newcommand{\g}{\gamma}
\newcommand{\D}{\Delta}
 
\newcommand{\sub}{\subseteq}
\newcommand{\Ex}{\mathbb{E}}

\newcommand{\N}{\mathbb{N}}
\newcommand{\HH}{{\cal H}}

\newcommand{\C}{\mu}

\DeclareMathOperator{\trace}{Tr}

\date{}

\title{Traces of Hypergraphs}
\author{
Noga Alon\thanks{Department of Mathematics, Princeton University, 
Princeton, New Jersey, USA and
Center for Mathematical Sciences and Applications, 
Harvard University, Cambridge, Massachusetts, USA.
Email: \texttt{nalon@math.princeton.edu}.}
\and
Guy Moshkovitz\thanks{Center for Mathematical Sciences and Applications, Harvard University, Cambridge, Massachusetts, USA. Email: \texttt{guymoshkov@gmail.com}.}
\and
Noam Solomon\thanks{Center for Mathematical Sciences and Applications, Harvard University, Cambridge, Massachusetts, USA. Email: \texttt{noam.solom@gmail.com}.}
}
\begin{document}
\maketitle

\begin {abstract} 
Let $\trace(n,m,k)$ denote the largest number of distinct projections onto $k$ coordinates guaranteed in any family of $m$ binary vectors of length $n$.
The classical Sauer-Perles-Shelah Lemma implies that $\trace(n, n^r, k) = 2^k$ for $k \le r$.
While determining $\trace(n,m,k)$ precisely for general $k$ and $m$ seems
hopeless, estimating it remains a widely open problem with connections to 
important questions in computer science and combinatorics.
For example, an influential result of Kahn-Kalai-Linial gives non-trivial bounds on	 $\trace(n, m, k)$ for $k=\Theta(n)$ and $m = \Theta(2^n)$.
Here we prove that, for $r,\a^{-1} \le n^{o(1)}$, it holds that
$\trace(n,n^r,\a n) = n^{\C(1+o(1))}$ with $$\C=\frac{r+1-\log(1+\a)}{2-\log(1+\a)}.$$
Thus, we (essentially) determine $\trace(n,m,k)$ for $k=\Theta(n)$ and all $m$ up to $2^{n^{o(1)}}$.


For the proof we establish a ``sparse'' version of another 
classical result, the Kruskal-Katona Theorem, 
which gives a stronger guarantee when the hypergraph does not 
induce dense sub-hypergraphs.
Furthermore, we prove that the parameters in our sparse Kruskal-Katona 
theorem are essentially best possible. 
Finally, we mention two simple applications which may be 
of independent interest. 
\end {abstract}

\section{Introduction}\label{se:intro}


For a hypergraph (or a set system) $\FF$, the \emph{trace} of $\FF$ on a vertex subset $I$ is defined as the set of projections of the edges of $\FF$ onto $I$, namely,
$\FF_I = \{e \cap I \,\colon\, e \in \FF \}$.
The \emph{shatter function}, or \emph{trace function}, of $\FF$ is 
$\trace(\FF,k) =\max_{I} \big|\FF_I\big|$ with $I$ a set $k$ vertices.
The focus of this paper is the following important extremal function; for integers $n \ge k$ and $0 \leq m \leq 2^n$, let $\trace(n,m,k)$ denote 
the largest number of distinct projections onto $k$ vertices guaranteed in any $n$-vertex $m$-edge hypergraph:
$$\trace(n,m,k) = \min_{\substack{\FF \sub 2^{[n]}\\|\FF|=m}} \trace(\FF,k) = \min_{\substack{\FF \sub 2^{[n]}\\|\FF|=m}} \, \max_{\substack{I \sub [n]\\|I|=k}} \big|\FF_I\big| .$$ 

There is a considerable number of results, in various areas of discrete mathematics, 
determining or estimating this function for certain values
of the parameters. The most famous result is arguably the Sauer-Perles-Shelah Lemma (\cite{Sa}, \cite{Sh}, see also
Vapnik and Chervonenkis~\cite{VC} for a slightly weaker estimate).
\begin {theo}[Sauer-Perles-Shelah]
\label{th:sps}
$\trace(n,m,k)=2^k$ for $m > \sum_{i=0}^{k-1} \binom{n}{i}$.
\end {theo}
The \emph{VC-dimension} of a hypergraph $\FF$ is the largest $k$ so that $\trace(\FF, k) = 2^k$, i.e., it is the largest number $k$ so that $\FF$ has a full projection on some $k$ vertices. The VC-dimension is a basic combinatorial measure of the \emph{complexity} of a hypergraph; understanding the shatter function beyond the case of full projections is a very natural direction. Shatter functions and the VC-dimension are extensively studied in combinatorial and computational geometry, as well as in machine learning (see the survey of Matou\v sek~\cite{Mat} for several geometric and algorithmic applications of shatter functions, and the survey of Angulin~\cite{An} for the role VC-dimension is playing in computational learning theory).

In~\cite{Bo}, Bondy proved that $\trace(n,n,n-1)=n$, and a
remark in \cite{Al} and in \cite{Fr} is that
$$
\trace(n,m,3)=7\, \text{ for } \,m=1+n+[n^2/4]+1,
$$
and the same argument implies that
determining the smallest $m$ for which $\trace(n,m,4)=15$ is equivalent
to determining the maximum possible number of edges of a
$3$-uniform hypergraph  on $n$ vertices with no complete hypergraph
on $4$ vertices---a well-known open problem of Tur\'an. Additional results
that can all be formulated in terms of the function $\trace(n,m,k)$ 
appear in~\cite{BR}, \cite{KKL}, \cite{BKK}, \cite{CGN} and more.

Recently, Bukh and Goaoc 
were able to estimate $\trace(n,n^{O(1)},k)$ for constant values of $k$ that are not too small (and also improved an earlier lower bound of \cite{CGN}).
%
In the other extreme regime, a classical paper of Kahn et al.~\cite{KKL} proves that for every $0<\a, \beta<1$,
$$\trace(n, \beta 2^{n}, \a n) \ge (1-n^{-c})2^{\a n},$$
where $c=c(\a,\beta)>0$ depends only on $\a,\beta$.
Benny Chor conjectured in the 80s that one can in fact make the error term exponentially rather than polynomially small in $n$.
This conjecture was recently disproved by Bourgain et al.~\cite{BKK}. In fact, Bourgain et al.\ prove several additional
results, in particular strengthening those of~\cite{KKL}.
%
%


In~\cite{BR}, Bollob\'as and Radcliffe considered the case where $m$ is polynomial and $k$ is linear. For the lower bound they were able to prove the following.
\begin{theo} [Bollob\'as and Radcliffe{~\cite[Theorem 7]{BR}}]
	\label{th:brlb}
	For constants $r\ge 2$ and $0 < \a \le 1$ it holds that $\trace(n,\,n^r,\,\a n) \ge \Omega(n^{\lambda r})$
	with\footnote{Here $H(x) = -x\log(x) - (1-x)\log(1-x)$ is the binary entropy function, and the logarithms are in base $2$.}
	$$\lambda =
	\begin{cases}
	\log(1+\alpha) & \alpha \in [\sqrt 2 -1, 1]\\
	\log(1+\alpha)/H(\log(1+\alpha)) & \alpha \in (0, \sqrt 2 -1)
	\end{cases}$$
\end{theo}

As for the upper bound, it would seem that among hypergraphs on $n$ vertices with a given number of edges $m$, a hypergraph $\FF$ with $\trace(\FF,k) = \trace(n,m,k)$ should be very symmetric, when $k$ is not too small or too large. A natural candidate for such an extremal hypergraph is thus the hypergraph containing all edges up to the appropriate size.
Bollob\'as and Radcliffe were able 
to show that this is in fact not the case. 
\begin {theo} [Bollob\'as-Radcliffe{~\cite[Theorem 11]{BR}}]
\label{th:brup}
For every constant integer $r \ge 2$, 
$$\trace\bigg(n,\, \sum_{i=0}^r \binom n i,\, n/2\bigg) \le o(n^r) .$$
\end {theo}



\subsection{Our results}
Our main result in this paper determines the value of $\trace(n,n^r,\a n)$, for constant $r$ and $\a$, up to logarithmic
factors, thus closing the gap between the lower and upper bounds in Theorems~\ref{th:brlb} and~\ref{th:brup}. 
We henceforth use the following standard notation: for two functions $f(n)$ and $g(n)$, by $f=\tilde{O}(g)$ we mean
$f = O(g\log^c(g))$ for some absolute constant $c>0$;
$f=\tilde{\Omega}(g)$ and
$f=\tilde{\Theta}(g)$ are defined analogously. 
The main result of this paper is as follows.
\begin{theo}[Main result]\label{t12}
	Let $r \ge 1$, $\a \in (0,1]$. If $r, \a^{-1} \le n^{o(1)}$ then 
	$\trace(n,\,n^r,\,\a n) = n^{\C(1-o(1))}$
	where
	\begin{equation}\label{eq:C}
	\C=\C(r,\a)=\frac{r+1-\log(1+\a)}{2-\log(1+\a)}.
	\end{equation}
	Moreover, if $r=O(1)$, $\a^{-1} \le (\log n)^{O(1)}$ then $\trace(n,\,n^r,\,\a n) = \tilde{\Theta}(n^{\C})$.
\end{theo}

It is perhaps instructive to consider one representative special 
case: $r=2$, $\a = 1/2$. 
In this case, 
the proofs in~\cite{BR} 
bound $\trace(n,n^2,n/2)$ as follows;
$$
\Omega(n^{1.169925..}) =
\Omega(n^{2 \log_2 {3/2}}) \leq \trace(n,n^2,n/2) 
\le \frac{n^2 (\log \log n)^{O(1)}}{\log n} = o(n^2),
$$
whereas Theorem~\ref{t12} in particular implies that
$$\trace(n,n^2,n/2)=\tilde{\Theta}(n^{1+1/(3-\log_2 3)})
=\tilde{\Theta}(n^{1.706695..}).
$$

\paragraph{A new Kruskal-Katona-type theorem.}

As it turns out, our main result can be readily deduced from a new version of the well-known Kruskal-Katona Theorem.
Recall that the Kruskal-Katona Theorem gives a lower bound on the number of $i$-sets contained within the edges of a uniform hypergraph. Formally, for a hypergraph $\FF$ and $i \in \N$ we denote 
$$\binom{\FF}{i} = \big\{S \,\big\vert\, |S|=i \text{ and } \exists e \in \FF \colon S \sub e \big\}.$$
The following classical version of the Kruskal-Katona Theorem 
was given by Lov\'asz~\cite{Lovasz}.
Henceforth, for real $y > 0$ we use the standard notation 
$$\binom{y}{i}=\frac{y(y-1)\cdots(y-i+1)}{i!} .$$ 
We use the abbreviation that $\FF$ is a \emph{$k$-graph} 
to mean that $\FF$ is a $k$-uniform hypergraph.

\begin{theo}[Kruskal-Katona Theorem, Lov\'asz~\cite{Lovasz}]
	\label{th:kk}
	Let $\FF$ be a $k$-graph. 
	If $|\FF| = \binom{y}{k}$ with real $y > 0$
	then for every $0 \le i \le k$ we have $\big|\binom{\FF}{i}
\big| \ge \binom{y}{i}$.
\end{theo}

Our new version of the Kruskal-Katona Theorem gives a stronger lower bound depending on the sparsity 
of the hypergraph $\FF$.
As is standard, we denote the sub-hypergraph of a hypergraph $\FF$ induced on a vertex subset $I$ by
$\FF[I] =(\,I, \,\,\{e\,\vert\, e \in \FF \text{ and } e\sub I\}\,)$.
We denote the largest number of edges in an induced sub-hypergraph on $i$ vertices by
$$\wp(\FF,i) := \max_{\substack{I \sub [n]\\|I|=i}} \big|\FF[I]\big|.$$
%
We next state our new version of the Kruskal-Katona Theorem. 
(See Theorem~\ref{th:skk} for a slightly stronger form.)
For the parameters relevant to our applications here, it provides
a significantly stronger estimate than the classical theorem,
using an appropriate sparseness assumption.

%
%


\begin{theo}[``Sparse Kruskal-Katona Theorem'']\label{th:skk-intro}
	Let $\FF$ be a $k$-graph with $n$ vertices and $|\FF|=n^r$ edges, $r \ge 1$. 
	If 
	$$\wp(\FF,\,\a n) \le \min\bigg\{\binom{x}{k-\ceil{r}}n, \,\, \frac{1}{2}|\FF|\bigg\}$$
	with real $x \ge 2k$ then for every $r+1 \le i \le k$ we have 
	$$\bigg|\binom{\FF}{i}\bigg| \ge \frac{1}{C}\cdot \frac{\binom{x}{i}}{\binom{x}{k}}|\FF|,$$
	with $C = (8k/\a)^{\ceil{5r}}\log n$.
\end{theo}


It is of course natural to ask whether the bound in Theorem~\ref{th:skk-intro} is essentially best possible.
Our third result in this paper proves that this is indeed the case.

\begin{theo}[Upper bound for Sparse Kruskal-Katona]\label{theo:sKK-UB-intro}
	Let $n,k,x \in \N^+$, $r \ge 1$ and $0 < \a \le 1$ with 
	$3r \le k \le x \le n^{1/6}$
	and $n \le \a^k n^r \le \binom{x}{k}n$.
	There exists a $k$-graph $\FF$ with $n$ vertices, $|\FF|=n^r$ edges, 
	and 
	$\wp(\FF,\,\a n) \le O(\binom{x}{k}n)$
	such that for every $0 \le i \le k$ we have 
	$\big|\binom{\FF}{i}\big| \le \frac{\binom{x}{i}}{\binom{x}{k}}|\FF|$.
\end{theo}

\paragraph{Applications.}
We end the paper with two simple 
applications of our main result, in geometry and in graph theory. 
We first describe the geometric application.
Let $\HH$ be a family of halfspaces\footnote{A halfspace consist of all points above a hyperplane.}  in $\R^d$, and let $P$ be a set of points in $\R^d$. We say that $P$ \emph{separates} $\HH$ if for every pair of distinct halfspaces $H_1 \neq H_2 \in \HH$ there is a point in $P$
that lies in one and outside the other.

\begin{prop}\label{prop:application-intro}
	Let $P \subset \R^d$ be a set of $n$ points and let $\HH$ be a family of $n^r$ halfspaces in $\R^d$, with $1 \le r \le n^\d$, such that $P$ separates $\HH$.  
	Then there exists a subset $P' \subseteq P$ of at most $n^{1-\d}$ points and
	a subset $\HH' \subseteq \HH$ of at least $n^{\frac{r+1}{2}(1-O(\d))}$ halfspaces such that $P'$ separates $\HH'$.
%
\end{prop}

The graph-theoretic application shows that in any graph with not too many independent sets, one can always find an induced subgraph on a vanishingly small number of vertices that nevertheless retains significantly more than square root of the total number of independent sets.

\begin{prop}
	Let $G = (V,E)$ be an $n$-vertex graph, and assume that the number of independent sets in $G$ is $n^r$ with $1 \le r \le n^\d$.
	Then there exists a subset $V' \subseteq V$ of at most $n^{1-\d}$ vertices such that the number of independent sets in the induced subgraph $G[V']$ is at least $n^{\frac{r+1}{2}(1-O(\d))}$.
\end{prop}

\paragraph{Organization.}
In Section~\ref{sec:sKK-LB} we prove Theorem~\ref{th:skk-intro} using 
an appropriate 
hypergraph decomposition method. 
We use it in Subsection~\ref{sec:traces} to deduce the lower bound in Theorem~\ref{t12}. 
In Subsection~\ref{sec:sKK-UB} we prove that the parameters of Theorem~\ref{th:skk-intro} are essentially best possible, and  in Subsection~\ref{subsec:UB-traces} we prove a matching upper 
bound for Theorem~\ref{t12}, 
using a probabilistic construction.
Our applications, Proposition~\ref{pr:tr} and Proposition~\ref{prop:app-graphs}, are given in Section~\ref{sec:applications}.


\paragraph{Proofs overview.}
For the proof of the sparse Kruskal-Katona Theorem (Theorem~\ref{th:skk-intro}, see also Theorem~\ref{th:skk} below) we proceed as follows. In the first part of the proof we apply a new approximate hypergraph decomposition method, relying on the sparseness of the input hypergraph, into  links. The decomposition is performed iteratively, in each step finding many vertices of high degree within the current link and restricting the next links to them. The final step of these iterations consists of ``cleaning'' each link by iteratively removing vertices of low degree. We then prove, using the sparseness of the hypergraph, that in fact most of the parts in our decomposition have few edges.  
In the second part of the proof we find, by applying the classical Kruskal-Katona Theorem, $i$-subsets within the edges of each (sub)link separately. We then argue that, since the links approximately decompose the hypergraph, we may essentially collect the $i$-subsets from all links without much overcounting.
The proof of our lower bound for traces (in Theorem~\ref{t12}) follows quite easily from the sparse Kruskal-Katona Theorem by applying it on the most ``popular layer'' of the hypergraph (i.e., the uniform hypergraph with the most edges contained in our hypergraph, which we may assume is down-closed) and projecting onto a random subset of $\a n$ vertices.

For the proof that the parameters in the sparse Kruskal-Katona Theorem are essentially best possible (Theorem~\ref{theo:sKK-UB} below)
we give a randomized construction of a uniform hypergraph whose edges contain few $i$-subsets. The construction is fairly simple: the union of a carefully chosen number of cliques on random subsets, such that it simultaneously holds that there are many cliques and yet they are nearly edge disjoint. We show in particular that the expected number of edges induced on subsets of $\a n$ vertices is sufficiently small so as to allow taking a union bound over all cliques. 
The proof of the upper bound for traces (in Theorem~\ref{t12}) follows by taking the down-closed hypergraph generated by the uniform hypergraph above, and then upper bounding the expected trace on a random subset of $\a n$ vertices.
\vspace{0.2cm}



\noindent
Throughout the paper we 
assume, whenever needed, that $n$ is sufficiently large. 
All logarithms are in base $2$ unless otherwise specified. 
To simplify the presentation we omit all floor and ceiling signs whenever these are not crucial. 


\section{Sparse Kruskal-Katona and Traces Lower Bound}\label{sec:sKK-LB}

In this section we prove Theorem~\ref{th:skk-intro}.
Henceforth, for a hypergraph $\FF$ on $V$ and for a vertex subset $I \sub V$ we denote by $\FF(I)$ the \emph{link} of $I$ in $\FF$, that is, 
$$\FF(I)=(\,V \setminus I,\,\,\{e \setminus I \,\vert\, I \sub e \in \FF \}\,).$$
Note that if $\FF$ is a $k$-graph then $\FF(I)$ is
a $(k-|I|)$-graph.
For a tuple $U$ of vertices in $V$ we denote by $|U|$ the number of 
distinct vertices in $U$, and by $\FF(U)$ the link $\FF(I)$ where 
$I$ is the set of (distinct) vertices in $U$ (and so $\FF(U)$ is a $(k-|U|)$-graph).

For the proof we will need several lemmas which we state and prove below.
We begin with the following simple ``hypergraph regularization'' lemma.
\begin{lemma}
	\label{l31}
	Every hypergraph $\FF=(V,E)$ has an induced sub-hypergraph $\FF'=(V',E')$ satisfying:
	\begin{enumerate}
		\item\label{item:reg-1}
		$|E'|/|V'| \geq |E|/|V|.$
		\item\label{item:reg-2}
		The degree of each vertex $v \in V'$ in $\FF'$ is at least
		$\frac{|E'|}{2|V'| \log|V|}$.
		\item\label{item:reg-3}
		$|E'| > |E|/2.$
	\end{enumerate}
\end{lemma}
\begin{proof} 
	Put $n=|V|$, $V_0=V$ and $E_0=E$. Starting with $i=0$, as long as the
	hypergraph $\FF_i=(V_i,E_i)$ in which $|V_i|=n-i$ does not satisfy~(\ref{item:reg-2}), let $V_{i+1}$ be the set obtained from $V_i$ by
	removing a vertex of minimum degree in $\FF_i$, and let $\FF_{i+1}$ be
	the induced subhypergraph on this set. 
	It is easy to see that $|E_{i+1}|/|V_{i+1}| > |E_i|/|V_i|$ and hence this process must terminate with a nonempty hypergraph $\FF_j=(V_j,E_j)$. Define $V'=V_j$, $E'=E_j$. Then~(\ref{item:reg-1}) holds as the quantity $|E_i|/|V_i|$ keeps increasing during the process, 
	(\ref{item:reg-2}) holds by the definition of $j$, and (\ref{item:reg-3}) holds since 
	\begin{align*}
	|E'| &= |E_j| \ge |E_0| \prod_{i=0}^{j-1} \Big(1-\frac{1}{2(n-i) \log n}\Big) \\
	&\geq |E_0| \Big(1-\sum_{i=0}^{j-1} \frac{1}{2(n-i) \log n}\Big)
	\geq |E|\Big(1-\frac{\ln n}{2\log n}\Big ) > \frac{|E|}{2}.
	\end{align*}
\end{proof}

\begin {lemma}
\label{lelink}
Let $\FF$ be a hypergraph on $n$ vertices with $\wp(\FF,i) \le \frac{|\FF|}{2}$, for some integer $0<i<n$.
Then $F$ has at least $i$ vertices of degree at least $\frac{|\FF|}{2n}$. 
%
%
%
\end {lemma}


\begin{proof}
	Let $I \sub V(F)$ denote the set of vertices of $\FF$ of degree at least $\frac{|\FF|}{2n}$.
	Then
	$|F[I]| > |\FF|-n \cdot \frac{|\FF|}{2n} =
	\frac{|\FF|}{2}.$
	By the assumption on $\wp(\FF,i)$ we thus have $|I| > i$.
	%
	%
\end{proof}

By an iterative application of Lemma~\ref{lelink} we obtain the following. 
\begin {lemma}
\label{lelinkcor}
Let $s \in \N^+$ and let $\FF$ be a hypergraph on $V$ with 
$\wp(\FF,i) \le \frac{|\FF|}{2^{s} |V|^{s-1}}$. 
Then there are at least $i^s$ $s$-tuples $U \in V^s$ with $|\FF(U)| \ge \frac{|\FF|}{(2|V|)^s}$.
\end{lemma}


\begin{proof}
Put $n=|V|$. For a tuple $U$ of vertices in $V$ we denote by $\FF^U$ the sub-hypergraph of $\FF$ on $V$ with edge set $\{e \in \FF \,\colon\, U \sub e\}$. Note that $|\FF^U|=|\FF(U)|$.
We proceed by induction on $s$, noting that the induction basis $s=1$ is Lemma~\ref{lelink}. 
For the induction step, let $\FF$ be as in the statement, and note that by the induction hypothesis there are at least $i^{s-1}$ $(s-1)$-tuples $U \in V^{s-1}$ with $|\FF^U|=|\FF(U)| \ge \frac{|\FF|}{(2n)^{s-1}}$.
Fix one such $U=(v_1,\ldots,v_{s-1})$ 
and apply Lemma~\ref{lelink} on the hypergraph $\FF^U$, noting that, as required,
$$\wp(\FF^U,i) \le \wp(\FF,i) \le \frac{|\FF|}{2^{s} n^{s-1}} \le \frac{|\FF^U|}{2},$$
where the first inequality uses the fact that $\FF^U$ is a sub-hypergraph of $\FF$ on $V$.
Thus, $\FF^U$ has at least $i$ vertices $v$ of degree at least $\frac{|\FF^U|}{2n} \ge \frac{|\FF|}{(2n)^{s}}$. 
This means that for each such $v$, the $s$-tuple $U'=(v_1,\ldots,v_{s-1},v)$ satisfies $|\FF(U')| = |\FF^{U'}| \ge \frac{|\FF|}{(2n)^{s}}$.
Going over all $i^{s-1}$ $(s-1)$-tuples $U$ in a similar fashion, we deduce that the total number of $s$-tuples $U'$ as above is at least $i^{s-1} \cdot i$.
This completes the induction step and the proof. 	
\end {proof}

The following lemma gives
a unified lower bound for the summation $\sum_{i=0}^k \binom{x}{i}\g^i$ that is independent of the ratio between $k$ and $x$.
See Section~\ref{sec:aux} in the Appendix for a proof of this lemma.
\begin{lemma}\label{lemma:113}
	For every $k \in \N^+$ and real $0 \le \g \le 1$, $x \ge k$ we have
	$$\sum_{i=0}^k \binom{x}{i}\g^i 
	\ge \frac14\bigg(\sum_{i=0}^k \binom{x}{i}\bigg)^{\log(1+\g)} .$$
	%
\end{lemma}

Finally, we have the following well-known bounds.
\begin{claim}\label{claim:e}
	We have $e^{-2x} \le 1-x \le e^{-x}$, where the upper bound holds for every real $x$ and the lower bound holds for every $0 \le x \le 1/2$.
\end{claim}
\begin{proof}
	The upper bound is well known, and the lower bound follows from it since we have $1-x = \big(1+\frac{x}{1-x}\big)^{-1} \ge e^{-\frac{x}{1-x}} \ge e^{-2x}$, where the last inequality uses $0 \le x \le 1/2$.
\end{proof}

%

\subsection{Sparse Kruskal-Katona Theorem}


In this subsection we prove Theorem~\ref{th:skk}, which is a more precise version of Theorem~\ref{th:skk-intro}.
First, we will need a lemma which extends the classical Kruskal-Katona Theorem~\ref{th:kk} by collecting $i$-subsets from the hypergraph's links.

\newcommand{\II}{\mathcal{I}}

\begin{lemma}\label{lemma:neat}
	Let $\FF$ be a $k$-graph on $V$, and let $t \in \N$.
	For every $t \le i \le k$ we have
	$$\bigg|\binom{\FF}{i}\bigg| \ge i^{-t} \sum_{U \in V^t} \binom{x_U}{i-|U|} $$
	where $x_U$ is given by $|F(U)| = \binom{x_U}{k-|U|}$.
\end{lemma}
\begin{proof}
	Put $\II=V^t$, and let $t \le i \le k$.
	Apply the Kruskal-Katona Theorem~(Theorem~\ref{th:kk}) on each link $\FF(U)$ with $U \in \II$. Since $\FF(U)$ is a $(k-|U|)$-graph and $0 \le i-|U|  \le k-|U|$ (using $|U|\le t \le i$ for the lower bound), Theorem~\ref{th:kk} implies that $\big|\binom{\FF(U)}{i-|U|}\big| \ge \binom{x_U}{i-|U|}$.
	Now, for every $U \in \II$ denote 
	$$\binom{\FF(U)}{i}^* = \bigg\{ f \cup U \,\bigg\vert\, f \in \binom{\FF(U)}{i-|U|} \bigg\}.$$ 
	We have that
	$$\binom{\FF(U)}{i}^* \sub \binom{\FF}{i} \quad\text{ and }\quad \bigg|\binom{\FF(U)}{i}^*\bigg| = \bigg|\binom{\FF(U)}{i-|U|}\bigg| \ge \binom{x_U}{i-|U|}.$$
	We therefore deduce that
	$$\bigg|\binom{\FF}{i}\bigg| \ge \bigg|\bigcup_{U \in \II} \binom{\FF(U)}{i}^*\bigg| \ge i^{-t}\sum_{U \in \II} \bigg|\binom{\FF(U)}{i}^*\bigg| \ge i^{-t}\sum_{U \in \II} \binom{x_U}{i-|U|} ,$$
	where, crucially, the penultimate inequality uses the fact that if an $i$-set $g$ appears in $\binom{\FF(U)}{i}^*$ then 
	$U \in g^t$, 
	implying that $g$ appears in at most $i^t$ families $\binom{\FF(U)}{i}^*$.
	This completes the proof.
	%
\end{proof}

%
%
Note that Lemma~\ref{lemma:neat} recovers Theorem~\ref{th:kk} by taking $t=0$.



We prove the following stronger form of Theorem~\ref{th:skk-intro}, our sparse Kruskal-Katona Theorem.

\begin{theo}\label{th:skk}
	Let $\FF$ be a $k$-graph with $n$ vertices and $|\FF|=n^r$ edges, $r \ge 1$. Let $s \in \N$ be the smallest satisfying $\frac{|\FF|}{(2n)^s} < c \cdot \wp(\FF,\, \a n)$ with 
	$c=(8k)^{\ceil{2r}}/\a^{\ceil{r}}$
	,\footnote{One may think of $|\FF|/(2n)^s$ as an approximation (from below) to the average degree of a vertex $s$-tuple in $\FF$. Alternatively, $s$ can be defined as $s=\big\lceil\log\big(\frac{|\FF|}{\sigma}\big)/\log(2n)\big\rceil$ with $\sigma=c \cdot \wp(\FF,\a n)$. 
	} 
	and put $t=s+1$.
	If 
	$$\wp(\FF,\,\a n) \le \min\bigg\{\binom{x}{k-t}n, \,\, \frac{1}{2}|\FF|\bigg\}$$
	with real $x>0$ then for every $t \le i \le k$ we  have 
	$$\bigg|\binom{\FF}{i}\bigg| \ge \frac{1}{C}\cdot \frac{\binom{x}{i-t}}{\binom{x}{k-t}}|\FF|,$$
	with $C = (8k/\a)^{\ceil{4r}} \log n$.
\end{theo}

\begin{remark}
	The parameters in Theorem~\ref{th:skk} satisfy the following relations:
	\begin{equation}\label{eq:sKK-relations}
		t \le \ceil{r} \le k.
	\end{equation}
	For the first inequality, note that otherwise $s=\ceil{r}$ and so, by the definition of $s$, $c \cdot \wp(\FF,\a n) \le |\FF|/(2n)^{s-1} \le n$, implying that $\wp(\FF,\a n) < \a n$ and thus, by averaging, $|\FF| < n$, contradicting the statement's assumption $r \ge 1$. 
	For the second inequality, notice $n^r = |\FF| \le \binom{n}{k} \le n^k$.
\end{remark}


Note that the error term $C$ increases with the quotient $k/\a$ and with $r$.
This precludes us from taking hypergraphs of large uniformity, with many edges, or from inducing on too few vertices.
More formally, we have the following corollary.

\begin{remark}
	Under the assumptions of Theorem~\ref{th:skk}: 
	\begin{enumerate}
		\item If $(k/\a)^r \le (\log n)^{O(1)}$ then $|\FF|/C \ge \tilde{\Omega}(|\FF|)$.
		\item If $k/\a \le n^{o(1)}$ then $|\FF|/C \ge |\FF|^{1-o(1)}$.
	\end{enumerate}
\end{remark}



We now show how to deduce the sparse Kruskal-Katona Theorem from Theorem~\ref{th:skk}. 
\begin{proof}[Proof of Theorem~\ref{th:skk-intro}]
	We have that $\binom{x}{k-\ceil{r}} \le \binom{x}{k-t}$ using~(\ref{eq:sKK-relations}) and since, by assumption, $x \ge 2k$. Thus, the condition here implies the condition in Theorem~\ref{th:skk}.
	As for the guarantee in Theorem~\ref{th:skk}, note that
	\begin{align*}
	\frac{\binom{x}{i-t}}{\binom{x}{k-t}} &= \frac{(k-t)!}{(i-t)!} \bigg(\prod_{j=i}^{k-1}(x-j+t)\bigg)^{-1}\\
	&= \frac{(k-t)!}{(i-t)!} \bigg(\prod_{j=i}^{k-1}(x-j) \left(1+\frac{t}{x-j}\right) \bigg)^{-1} 
	\ge \frac{k^{-t}k!}{i!} \bigg(\bigg(1+\frac{t}{x-(k-1)}\bigg)^{k-i} \cdot \prod_{j=i}^{k-1} (x-j) \bigg)^{-1}\\
	&\ge \frac{k^{-t}k!}{i!} \bigg(\Big(1+\frac{t}{k}\Big)^{k} \cdot \prod_{j=i}^{k-1} (x-j) \bigg)^{-1}
	\ge (ek)^{-t}\frac{k!}{i!} \bigg(\prod_{j=i}^{k-1} (x-j) \bigg)^{-1}
	= (ek)^{-t}\frac{\binom{x}{i}}{\binom{x}{k}} ,
	\end{align*}
	where the second inequality uses the statement's assumption $x \ge 2k$, and the third inequality uses the upper bound in Claim~\ref{claim:e}.
	Thus, multiplying $C$ from Theorem~\ref{th:skk} by $(8k)^{\ceil{r}} 
	\ge (ek)^t$ (recall~(\ref{eq:sKK-relations})) completes the proof.
	%
\end{proof}


\newcommand{\reg}{_\text{reg}}

\begin{proof}[Proof of Theorem~\ref{th:skk}]
	We will prove the implication that if the stronger condition
	\begin{equation}\label{as:st}
	\wp(\FF,\,\a n) \le \min\bigg\{\frac{1}{c}\binom{x}{k-t}\a n,\,\, 
\frac12|\FF|\bigg\},
	\end{equation} 
	(where $c$ is as in the statement of the theorem) holds then for every $t \le i \le k$ we in fact have
	\begin{equation}\label{eq:sKK-first-goal}
	\bigg|\binom{\FF}{i}\bigg| \ge \frac{1}{c\log n}|\FF|\frac{\binom{x}{i-t}}{\binom{x}{k-t}} .
	\end{equation}
	To see why this would complete the proof, let $\tilde x$ satisfy 
	\begin {equation}
	\label {eq:tildex}
	\binom{\tilde x}{k-t} = \frac c \a \binom{x}{k-t},
	\end {equation} 
	so that if $\FF$ satisfies the statement's original assumption that $\wp(\FF,\,\a n) \le \min\big\{\binom{x}{k-t}n,\,\, \frac{1}{2}|\FF|\big\}$ then it satisfies~(\ref{as:st}) with $\tilde x$ replacing $x$.
	Thus, from~(\ref{eq:sKK-first-goal}),
	$$\bigg|\binom{\FF}{i}\bigg| \ge \frac{1}{c\log n}|\FF|\frac{\binom{\tilde x}{i-t}}{\binom{\tilde x}{k-t}} 
	\ge \frac{\a}{c^2\log n}|\FF|\frac{\binom{x}{i-t}}{\binom{x}{k-t}} 
	\ge \frac{1}{C}|\FF|\frac{\binom{x}{i-t}}{\binom{x}{k-t}} ,$$
	where the second inequality uses~(\ref{eq:tildex}) for both the denominator and the
numerator, as~(\ref{eq:tildex}) 
implies $\tilde x \ge x$ since $c/\a \ge 1$.
	
	We henceforth assume~(\ref{as:st}), and our goal is to prove~(\ref{eq:sKK-first-goal}).
	By the definition of $s$ we have 
	\begin{equation}\label{eq:sKK-s}
	\frac{|\FF|}{(2n)^s} < c \cdot \wp(\FF,\, \a n) \le \frac{|\FF|}{(2n)^{s-1}}.
	\end{equation}
	The upper bound in~(\ref{eq:sKK-s}) implies in particular that $\wp(\FF,\, \a n) \le \frac{|\FF|}{2^{s} n^{s-1}}$.
	Thus, by Lemma~\ref{lelinkcor}, there is a family $\UU \sub V(F)^s$ of $s$-tuples 
	$U$ of vertices of $\FF$ satisfying 
	\begin{equation}\label{eq:link}
	|\FF(U)| \ge \frac{|\FF|}{(2n)^s} =: b ,
	\end{equation}
	such that $|\UU| \ge \a^s n^s$.
	For each $U \in \UU$ apply Lemma~\ref{l31} on the $(k-|U|)$-graph $\FF(U)$ to obtain an induced subgraph $\FF(U)\reg$ 
	with
	\begin{equation}\label{eq:link-edges}
	|\FF(U)\reg| > \frac12|\FF(U)| \ge \frac12 b,
	\end{equation}
	such that $\FF(U)\reg$ has $n_U$ vertices and minimum degree at least
	\begin{equation}\label{eq:F_U}
	\frac{1}{4 \log n} \cdot \frac{b}{n_U} =: \binom{x_U}{k-(|U|+1)}
	\end{equation}
	with $x_U>0$.
	Let $t \le i \le k$.
	Applying Lemma~\ref{lemma:neat} with our $t$ to the links of the vertices of each $\FF(U)\reg$, we deduce that
	\begin{align}\label{eq:sKK-main}
	\bigg|\binom{\FF}{i}\bigg| &\ge \frac{1}{i^t}\sum_{U \in \UU} n_U\binom{x_U}{i-(|U|+1)} 
	\ge \frac{1}{4(2k)^t\log n} \frac{1}{n^s}\sum_{U \in \UU}\frac{\binom{x_U}{i-(|U|+1)}}{\binom{x_U}{k-(|U|+1)}} |\FF| ,
	\end{align}
	where the second inequality uses~(\ref{eq:F_U}) together with~(\ref{eq:link}) and the fact that $i \le k$.
	%
	%
	Let 
	$$\UU' = \bigg\{ U \in \UU \,\colon\, \binom{x_U}{k-(|U|+1)} \le \binom{x}{k-t} \bigg\}.$$
	%
	%
%
	For every $U \in \UU'$ we have $\frac{\binom{x_U}{i-(|U|+1)}}{\binom{x_U}{k-(|U|+1)}} \ge k^{-t} \frac{\binom{x}{i-t}}{\binom{x}{k-t}}$,
	which follows from Claim~\ref{claim:binom-ratio}.\footnote{Indeed, take $x$, $y$, $k$, $i$, $\D$ there to be, respectively, $x$, $x_U$, $k-(|U|+1)$, $i-(|U|+1)$,  $t-(|U|+1)$, and bound $(i-(|U|+1))^{-(t-(|U|+1))}$ from below by $k^{-t}$.}
%
	%
	%
	We will show that
	\begin{equation}\label{eq:goal}
	|\UU'| \ge \frac12 \a^s n^s .
	\end{equation}
%
%
%
	%
	By~(\ref{eq:sKK-main}), this would imply that
	$$\bigg|\binom{\FF}{i}\bigg| \ge \frac{\a^s}{8(2k)^{2t}\log n} \frac{\binom{x}{i-t}}{\binom{x}{k-t}} |\FF|
	\ge \frac{1}{c\log n}\frac{\binom{x}{i-t}}{\binom{x}{k-t}} |\FF| ,$$
	where the last inequality uses $8(2k)^{2t}/\a^s \le (8k)^{\ceil{2r}}/\a^{\ceil{r}} = c$
	(recall~(\ref{eq:sKK-relations})).
	Thus, proving~(\ref{eq:goal}) would imply~(\ref{eq:sKK-first-goal}) and complete the proof.
	
	
	Put $S = {\binom{x}{k-t}} \a n$. 
	It remains to prove~(\ref{eq:goal}). 
	Assume for contradiction that $|\UU \setminus \UU'| \ge \frac12(\a n)^s$.
	Note that by definition of $\UU'$ together with~(\ref{eq:F_U}) 
we deduce that for every $U \in \UU \setminus \UU'$ we have 
	$$
n_U \le \frac{1}{4 \log n} \cdot \frac{b}{\binom{x}{k-t}} \leq
\frac12\cdot\frac{b}{\binom{x}{k-t}} = \frac{b}{S} \cdot \frac12\a n =: n_0.
$$
	Note that from the lower bound in~(\ref{eq:sKK-s}) together with~(\ref{as:st}) we deduce that $b \le S$.
	Since $n_0 \ge n_U \ge 1$, we have that $n_0$ satisfies
	\begin{equation}\label{eq:sKK-n0-bounds}
	1 \le n_0 \le \frac12 \a n .
	\end{equation}
	Put
	\begin{equation}\label{eq:sKK-ell}
	\ell = \frac12\cdot\begin{cases}
	\frac{\a n}{n_0+s} 	& \text{ if } s \ge 1\\
	1					& \text{ if } s = 0
	\end{cases}
	\end{equation}
	%
	Note that~(\ref{eq:sKK-n0-bounds}) implies 
	\begin{equation}\label{eq:sKK-l-bounds}
	1 \le \frac{\a n}{n_0+s} \le \a n ,
	\end{equation}
	where the lower bound further uses the fact that $s \le r$ and the fact that we may assume $r \le \a n/2$ as otherwise there is nothing to prove\footnote{Otherwise $C \ge |\FF|$, and so the statement's lower bound on $\big|\binom{\FF}{i}\big|$ is trivially true since $\frac{\binom{x}{i-t}}{\binom{x}{k-t}} \le \frac{\binom{k-t}{i-t}}{\binom{k-t}{k-t}} \le \binom{k}{i}$, where the first inequality uses the decreasing monotonicity of the function $z \mapsto \binom{z}{a}/\binom{z}{b}$ with $a \le b \le z$.}.
	Let $\UU^* \sub \UU \setminus \UU'$ be an arbitrary subset with $|\UU^*|=\ceil{\ell}$, 
	which is well defined as $\ell \le \frac12 (\a n)^s \le |\UU \setminus \UU'|$; here, the first inequality is immediate for $s=0$ by~(\ref{eq:sKK-ell}), and for $s\ge 1$ follows from the upper bound in~(\ref{eq:sKK-l-bounds}) together with the bound $\a n \le (\a n)^s$.
	For each $U \in \UU$ denote $I_U = V\big(\FF(U)\reg\big) \cup U$, and note that
	\begin{equation}\label{eq:liin}
	\FF[I_U] = \big\{e \cup U \mid e  \in \FF(U)\reg\big\}.
	\end{equation}
	Let $I = \bigcup_{U \in \UU^*} I_U$ denote the union of these sets of vertices. Then $I$ satisfies that
	$$|I| \le \sum_{U \in \UU^*} (n_U+s) \le \ceil{\ell} (n_0+s) \le 2\ell(n_0+s) \le \a n ,$$
	where the penultimate inequality uses the lower bound $\ell \ge \frac12$ from~(\ref{eq:sKK-l-bounds}).
	Moreover, $I$ satisfies that
	$$|\FF[I]| \ge \Big|\bigcup_{U \in \UU^*} \FF[I_U] \Big|
	\ge k^{-s} \sum_{U \in \UU^*} \big|\FF[I_U]\big|
	= k^{-s} \sum_{U \in \UU^*} \big|\FF(U)\reg\big|
	> k^{-s} \cdot \frac12 b \ceil{\ell},$$
	where the second inequality uses the fact that
	$e \in \FF[I_U]$ for at most $k^s$ $s$-tuples $U$ 
	(by~(\ref{eq:liin}), $e \in \FF[I_U]$ implies $U \sub e$), and the third inequality uses~(\ref{eq:link-edges}).
	Now, if $s=0$ then $\ceil{\ell}=1$ and $b=|\FF|$, hence we get $|\FF[I]| > |\FF|/2 \ge \wp(\FF,\a n)$ using~(\ref{as:st}), a contradiction.
	Otherwise, we get
	$$|\FF[I]| > \frac{1}{8 k^{s+1}} \cdot b\frac{\a n}{n_0}
	= \frac{1}{4k^{s+1}} S \ge \frac{1}{c}S \ge \wp(\FF,\,\a n),$$
	where the first inequality uses~(\ref{eq:sKK-ell}) and bounds $n_0+s \le 2n_0k$ (as $n_0 \ge 1$ by~(\ref{eq:sKK-n0-bounds}) and $s \le k$ by~(\ref{eq:sKK-main})), 
	the equality uses~(\ref{eq:sKK-n0-bounds}), 
	and the last inequality uses~(\ref{as:st}).
	We thus again obtain a contradiction.
	%
	This completes the proof. 
\end{proof}


We have the following important corollary of Theorem~\ref{th:skk}.

\begin{coro}
	\label{co:exptr}
	Let $\FF$ be a $k$-uniform hypergraph with $n$ vertices and $|\FF|=n^r$ edges, $r \ge 1$. 
	If 
	$$\wp(\FF,\,\a n) \le \min\Big\{B n,\,\, \frac12|\FF|\Big\}$$
	with real $B>0$, then for every $0 \le \gamma \le 1$ the expected trace of $\FF$ on a 
	uniformly random subset
	of $\g n$ vertices is at least
	$$\frac{1}{C'} \cdot \frac{|\FF|}{B^{1-\log(1+\gamma)}},$$
	with $C' = (8k/\a\g)^{\ceil{5r}}\log n$, provided $k \le \sqrt{\g n}$.
	\end{coro}
	
	\begin {proof}
	Write $B = \binom{x}{k-t}$ with $x>0$ with $t \in \N$ as in Theorem~\ref{th:skk}, so that 
	$\wp(\FF,\,\a n) \le \min\big\{\binom{x}{k-t}n,\,\, \frac{1}{2}|\FF|\big\}$.
	The expected trace of $\FF$ on a uniformly random set of $q = \g n$ vertices is, by Theorem~\ref{th:skk},
	\begin{align*}
	\sum_{i=0}^k \bigg| \binom{\FF}{i}\bigg| \frac{\binom{n-i}{q-i}}{\binom{n}{q}} &\ge \sum_{i=0}^k \bigg| \binom{\FF}{i}\bigg| \g^i e^{-1}
	\ge \frac{|\FF|}{eC\binom{x}{k-t}}\sum_{i=t}^k \binom{x}{i-t} \g^i \\
	&= \frac{|\FF|}{eC\binom{x}{k-t}}\g^t\sum_{j=0}^{k-t} \binom{x}{j} \g^{j} 
	\ge \frac{\g^t |\FF|}{4eC\binom{x}{k-t}}\binom{x}{k-t}^{\log(1+\g)}
	= \Big(\frac{\g^t}{4eC}\Big)\frac{|\FF|}{B^{1-\log(1+\g)}}
	\end{align*} 
	with $C= (8k/\a)^{\ceil{4r}}\log n$ as in Theorem~\ref{th:skk},
	where the first inequality uses (recall the lower bound in Claim~\ref{claim:e} and the statement's assumption $k \le \sqrt{q}$)
	$$\frac{\binom{n-i}{q-i}}{\binom{n}{q}} = \frac{(n-i)!}{n!}\frac{q!}{(q-i)!} = \prod_{j=0}^{i-1} \frac{q-j}{n-j}
	\ge \Big(\frac{q}{n}\Big)^i \prod_{j=0}^{i-1} (1-j/q)
	\ge \g^i \prod_{j=0}^{i-1} e^{-2j/q} = \g^i e^{-i^2/q} \ge \g^i e^{-1},$$
	%
	and the last inequality uses Lemma~\ref{lemma:113}. 
	Using~(\ref{eq:sKK-relations}) to bound $t$, the proof follows.
	\end {proof}
	
	Note that the statement in Corollary~\ref{co:exptr} does not depend on $k$, the uniformity of the hypergraph, except in the error term $C'$. In fact, in order for this statement to be meaningful, $C'$ should be negligible, so $k$ should be relatively small, e.g., poly-logarithmic in $n$.
	
	\begin{remark}
	If the bound in Corollary~\ref{co:exptr} is tight, then, in the proof of Theorem~\ref{th:skk}, all the inequalities become (essentially) equalities. In the proof of Theorem~\ref{th:skk}, we showed that the number of $s$-tuples in $\UU'$ is at least $\frac12(\a n)^s$.
	It is then easy to verify that for half of these $s$-tuples $U$, we have $t_U \approx t_0, x_U \approx x$ and then $x \approx 2k$. (Quantifying this statement precisely requires more details, which we omit here.)
	Note that this remark applies to Corollary~\ref{co:exptr} but not to Theorem~\ref{th:skk} itself, as evident from the matching upper bound in Subsection~\ref{sec:sKK-UB}.
	\end{remark}


\subsection{Tight lower bound for traces}\label{sec:traces}

Recalling the definition of $\C$ in~(\ref{eq:C}), we now prove the lower bound in our main result Theorem~\ref{th:main}.

\begin{theo} 
	\label{th:main}
	Let $r \ge 1$, $\a \in (0,1]$. If $r, \a^{-1} \le n^{\d}$ then
	$\trace(n,\,n^r,\,\a n) \ge n^{\C(1-O(\d))}$.
	Moreover, if $r=O(1)$, $\a^{-1} \le (\log n)^{O(1)}$ then $\trace(n,\,n^r,\,\a n) \ge \tilde\Omega(n^\C)$.
\end{theo}
\begin{proof}
	Let $\HH$ be a down-closed hypergraph\footnote{Also called monotone. A hypergraph $\HH$ is down-closed if for every edge $\HH$, all its subsets are also edges of $\HH$.} on $n$ vertices with $|\HH|=n^r$.
	We will prove lower bounds on $\trace(\HH,\a n)$, which would complete the proof since we may assume the given hypergraph is down-closed (this is standard, see, e.g.,~\cite{Al}).
	Put $\HH_i = \big\{e \in \HH \,\big\vert\, |e|=i\big\}$, and let $\FF=\HH_k$ where $k$ maximizes $|\HH_k|$.
	Observe that since $\HH$ is down-closed we have $|\HH| \ge 2^k$, which implies that $k \le \log|\HH|$, and therefore $|\FF| \ge |\HH|/(\log|\HH|+1)$.
	We assume $\trace(\FF,\,\a n) \le \frac12|\FF|$, as otherwise we are done since
	$$\trace(\HH,\a n) \ge \trace(\FF,\a n) \ge \frac12|\HH|/(\log|\HH|+1) = \tilde{\Omega}(n^r).$$ 
	
	Now, write $\trace(\FF,\,\a n) =B n$ with $B > 0$.
	Since $\wp(\FF,\,\a n) \le \trace(\FF,\,\a n)$ and since $k \le \sqrt{\a n}$ (i.e., $k^2/\a \le n$) for all large enough $n$, we apply Corollary~\ref{co:exptr} with $\g=\a$ to obtain that
	$$B n = \trace(\FF,\,\a n) 
	\ge \frac{1}{C'}\frac{|\FF|}{B^{1-\log(1+\a)}},$$
	with $C' =  (8k/\a^2)^{\ceil{5r}}\log n$, where the inequality uses the fact that the maximal trace on $\alpha n$ vertices is at least as large as the expected trace on a random subset 
	of $\a n$ vertices. 
	This gives a lower bound on $B$; we therefore deduce
	\begin{align*}
	\trace(\HH,\,\a n) &\ge \trace(\FF,\,\a n) = B n 
	\ge \Big(\frac{1}{C'}|\FF|\cdot \frac{1}{n}\Big)^{\frac{1}{2-\log(1+\a)}} \cdot n\\
	&= \Big(\frac{1}{C'}|\FF| n^{1-\log(1+\a)}\Big)^{\frac{1}{2-\log(1+\a)}}
	\ge \Big(\frac{1}{C''}|\HH|n^{1-\log(1+\a)}\Big)^{\frac{1}{2-\log(1+\a)}} 
	\ge \frac{1}{C''} \cdot n^{\C},
	\end{align*}
	with $C''=(\log|\HH|+1)C' \le (8\log|\HH|/\a^2)^{\ceil{5r}}\log n \le (8r\log n/\a^2)^{\ceil{6r}}$. 
	Now, if $r,\a^{-1} \le n^{\d}$ then $C'' \le n^{O(\d r)}$, which implies that $\trace(\HH,\,\a n) \ge n^{\C(1-O(\d))}$.
	Moreover, if $r \le O(1)$, $\a^{-1} \le (\log n)^{O(1)}$ then $C'' \le (\log n)^{O(1)}$, which implies that 
	$\trace(\HH,\,\a n) \ge \tilde\Omega(n^{\C})$.
	This completes the proof.
\end{proof}

\begin{remark*}
	In the proof of Theorem~\ref{th:main} it seems tempting to write $\trace(\FF,\alpha_0 n) = B_0 n$ with $\alpha_0 \ll \alpha$, so that $B_0 \ll B$. Then, Corollary~\ref{co:exptr} would mean that the expected trace on $\alpha n$ random vertices would be roughly $\frac{|\FF|}{B_0^{1-\log(1+\a)}} \gg \frac{|\FF|}{B^{1-\log(1+\a)}},$ which seemingly contradicts the fact that our lower bound is tight!
	The reason this cannot happen is that, for any $\alpha_0$ that is not too small, though the \emph{average} trace on $\alpha_0 n$ vertices is substantially smaller than that on $\alpha n$ vertices, the \emph{maximal} traces can actually be the same. Indeed, in the upper bound construction one can show that $\trace(\FF,\alpha_0 n) \approx \trace(\FF,\alpha n)$ for any $\alpha_0 \le \alpha$ that is not too small.
\end{remark*}


\section{Upper Bounds}

\subsection{Upper Bound for the Sparse Kruskal-Katona Theorem}\label{sec:sKK-UB}

In this subsection we show that the parameters in Theorem~\ref{th:skk} are best possible up to the error term. Formally, we prove the following. 

\begin{theo}[Upper bound for sparse Kruskal-Katona]\label{theo:sKK-UB}
	Let $n,k,x \in \N^+$, $r \ge 1$ and $0 < \a \le 1$ with 
	$3r \le k \le x \le n^{1/6}$
	and $n \le \a^k n^r \le \binom{x}{k}n$.
	There exists a $k$-graph $\FF$ with $n$ vertices, $|\FF|=n^r$ edges, 
	and 
	$\wp(\FF,\,\a n) \le 6\binom{x}{k}n$
	such that for every $0 \le i \le k$ we have 
	$\big|\binom{\FF}{i}\big| \le \frac{\binom{x}{i}}{\binom{x}{k}}|\FF|$.
\end{theo} 

We will need the following lemma relating hypergeometric and binomial random variables.
\begin{lemma}\label{lemma:HG}
	Let $H$ be a hypergeometric random variable with parameters 
$(n,x,y)$,\footnote{I.e., $H=|X \cap Y|$ where $X \sub [n]$ is a 
uniformly random subset of size $x$ and $Y \sub [n]$ is a fixed 
subset of size $y$.} 
	and let $B$ be a binomial random variable with parameters $(x,\, y/n)$.
	If $x \le \sqrt{n}$ then for every $0 \le h \le x$ we have $\Pr[H=h] \le 2\Pr[B=h]$.
\end{lemma}
\begin{proof}
	We have
	\begin{align*}
	\Pr[H=h] &= 
	\frac{\binom{y}{h}\binom{n-y}{x-h}}{\binom{n}{x}} 
	=\frac{\binom{x}{h}\binom{n-x}{y-h}}{\binom{n}{y}}
	= \binom{x}{h}\frac{(n-x)!}{n!}\cdot\frac{y!}{(y-h)!}\cdot\frac{(n-y)!}{(n-y-(x-h))!}\\
	&= \binom{x}{h}\frac{\prod_{t=0}^{h-1} y-t}{\prod_{t=0}^{h-1} n-t} 
	\cdot \frac{\prod_{t=0}^{x-h-1} n-y-t}{\prod_{t=0}^{x-h-1} n-h-t}
	\le \binom{x}{h}\bigg(\frac{y}{n}\bigg)^h\bigg(\frac{n-y}{n-h}\bigg)^{x-h} \\
	&\le \binom{x}{h}\bigg(\frac{y}{n}\bigg)^h\bigg
	(\frac{n-y}{n}\bigg)^{x-h} e^{(2h/n)\cdot(x-h)}
	\le 2\binom{x}{h} \bigg(\frac{y}{n}\bigg)^h\bigg
	(\frac{n-y}{n}\bigg)^{x-h}.
	\end{align*}
	where the first inequality uses that $h \le y$ (as otherwise 
$\Pr[H=h]=0$ and there is nothing to prove), the second inequality 
uses the lower bound in Claim~\ref{claim:e}
as $h \le x \le \sqrt{n} \le n/2$ (the last inequality assumes $\sqrt{n} \ge 2$, for otherwise $x \le 1$ in which case $H=B$ so there is nothing to prove),
and the third inequality uses the fact that $h(x-h) \le x^2/4 \le n/4$
	as $x \le \sqrt{n}$. 
	Recalling that $B$ is a binomial random variable with parameters $(x,\a)$ with $\a=y/n$, we deduce
	$$\Pr[H=h] \le 2\binom{x}{h} \a^h(1-\a)^{x-h} = 2\Pr[B=h] ,$$
	as desired.
\end{proof}

We will make use of the following version of Chernoff's bound
(c.f., e.g., \cite{AS}, Appendix A).
\begin{claim}\label{claim:Chernoff}
	Let $X_1,\ldots,X_n$ be mutually independent random variables with $X_i \in [0,1]$,
	and put $X=\sum_{i=1}^n X_i$, $\mu = \Ex[X]$.
	Then $\Pr(X \ge 6x) \le \exp(-x)$ for every $x \ge \mu/3$.
\end{claim}
\begin{proof}
	Put $y = 3x \ge \mu$. Then
	$\Pr(X \ge 6x) = \Pr(X \ge 2y) \le \Pr(X-\mu \ge y) \le \exp(-y/3) = \exp(-x),$
	where the second inequality is Chernoff's large-deviation bound (again using $y \ge \mu$).
\end{proof}

\begin{proof}[Proof of Theorem~\ref{theo:sKK-UB}]
	Let $\ell \in \N$ satisfy
	\begin{equation}\label{eq:UB-identity-KK}
	\frac{n^r}{\binom{x}{k}} \le \ell \le \frac{n}{\a^k},
	\end{equation}
	which is well defined by the statement's upper bound on $\a^k n^r$.
	%
	%
	Let $S_1,\ldots,S_\ell$ be $\ell$ independent uniformly random size-$x$ subsets of $[n]$.
	We let $\FF$ be the $k$-graph on $[n]$ consisting of a complete $k$-graph on each $S_j$, that is, $\FF = \big([n],\,\bigcup_{j=1}^\ell \binom{S_j}{k}\big)$.

	We next analyze the random $k$-graph $\FF$ constructed above.
	Let $E_1$ be the event that every two sets $S_j,S_{j'}$ $(1 \le j\neq j' \le \ell)$ intersect in fewer than $t:=3r$ elements, 
	and let $E_2$ be the event that $\wp(\FF,\, \a n) \le 6\binom{x}{k}n$. 
	We first show that the proof would follow by proving that 
	\begin{equation}\label{eq:KK-goal}
	\Pr(E_1 \text{ and } E_2) > 0 .
	\end{equation}
	To see this first note that, by construction, $\big|\binom{\FF}{i}\big| \le \sum_{j=1}^\ell \binom{|S_j|}{i} = \ell\binom{x}{i}$ for every $0 \le i \le k$.
	Furthermore, note that the event $E_1$ implies that the cliques $\binom{S_j}{k}$ are (edge-)disjoint by the statement's assumption $k\ge t$, and so
	\begin{equation}\label{eq:UB-disj}
	E_1 \text{ implies } |\FF|=\ell \binom{x}{k}.
	\end{equation}
	Therefore, (\ref{eq:KK-goal}) implies the existence of an $n$-vertex $k$-graph $\FF$ satisfying:
	\begin{itemize}
		\item $|\FF| = \ell \binom{x}{k} \ge n^r$, using~(\ref{eq:UB-disj}) and the lower bound in~(\ref{eq:UB-identity-KK}),
		\item $\wp(\FF,\, \a n) \le 6\binom{x}{k}n$, and
		\item $\big|\binom{\FF}{i}\big| \le \ell\binom{x}{i} = \frac{\binom{x}{i}}{\binom{x}{k}}|\FF|$, using~(\ref{eq:UB-disj}),
	\end{itemize}
	from which the proof immediately follows by taking an arbitrary subgraph of $\FF$ with $n^r$ edges.
	
	
	To prove~(\ref{eq:KK-goal}) we first claim that
	\begin{equation}\label{eq:F-size-KK}
	\Pr(E_1) \ge \frac12 .
	\end{equation}
	Denote by $B$ the binomial random variable with parameters $(x,p)$ where $p=x/n$. 
	For every $j \neq j'$ we have
	%
	\begin{align*}
	\Pr\big(|S_j \cap S_{j'}| \ge t\big) &\le 2\Pr\big(B \ge t\big) \le 2\binom{x}{t}p^{t} \\
	&\le (xp)^{t} = x^{2t}/n^t = x^{6r}/n^{3r} \le n^{- 2r} \le \ell^{-2},
	\end{align*}
	where the first inequality uses Lemma~\ref{lemma:HG} using the statement's upper bound on $x$,
	the penultimate inequality again uses the statement's upper bound on $x$,
	and the last inequality uses the upper bound in~(\ref{eq:UB-identity-KK}) together with the statement's lower bound on $\a^k n^r$.
	This implies, by taking the union bound over all $\binom{\ell}{2}$ unordered pairs $1 \le j\neq j' \le \ell$, that with probability at least $\frac12$ all set pairs $S_j,S_{j'}$ with $j \neq j'$ intersect at fewer than $3r$ elements. This proves~(\ref{eq:F-size-KK}).

	We now show that $\Pr(E_2) \ge 1/2$ (that is, that $\wp(\FF,\,\a n) \le 6\binom{x}{k}n$ except with probability smaller than $1/2$), which would prove~(\ref{eq:KK-goal}). 
	Fix $I \sub [n]$ of size $q=\a n$.
	Note that $\FF[I] = \bigcup_{j=1}^\ell \binom{S_j \cap I}{k}$.
	For each $1 \le j \le \ell$ consider the random variable $X_j = \binom{|S_j \cap I|}{k}$, let $X=\sum_{j=1}^\ell X_j$, and note that $|\FF[I]| \le X$.
	We have
	\begin{align*}
	\Ex(X_j) &= \sum_{h=k}^x \binom{h}{k}\Pr[|S_j \cap I|=h] 
	\le 2\sum_{h=k}^x\binom{x}{h}\binom{h}{k} \a^h(1-\a)^{x-h}\\
	&= 2\sum_{h=k}^x\binom{x}{k}\binom{x-k}{h-k} \a^h(1-\a)^{x-h}
	= 2\a^k\binom{x}{k}\sum_{j=0}^{x-k}\binom{x-k}{j} \a^j(1-\a)^{(x-k)-j}
	= 2\a^k\binom{x}{k},
	\end{align*}
	where the first inequality follows from Lemma~\ref{lemma:HG} using the statement's upper bound on $x$.
	Thus, by linearity of expectation,
	\begin{equation}\label{eq:Ex-KK}
	\Ex(X) \le \ell \cdot 2\a^k \binom{x}{k} \le 2\binom{x}{k}n,
	\end{equation}
	where the second inequality uses the upper bound in~(\ref{eq:UB-identity-KK}).
	%
	%
	Since $X_j \le \binom{x}{k}$ for every $1 \le j \le \ell$, 
	we have that $X/\binom{x}{k}$ is a sum of mutually independent random variables each in $[0,1]$.
	We thus apply Claim~\ref{claim:Chernoff} on $X/\binom{x}{k}$, using the fact that $n \ge \frac13\Ex(X/\binom{x}{k})$ by~(\ref{eq:Ex-KK}), to deduce that
	$$\Pr\Big[X \ge 6\binom{x}{k}n\Big] = \Pr\Big[X/\binom{x}{k} \ge 6n\Big] \le \exp(-n) < 2^{-n}.$$
	Using the union bound over all $\binom{n}{\a n} \le \frac12 2^n$ choices of $I \sub [n]$ with $|I|=\a n$ we deduce 
	that, except with probability smaller than $1/2$, for every $I \sub [n]$ with $|I|=\a n$ it holds that $|\FF[I]| \le 6\binom{x}{k}n$. As mentioned before, together with~(\ref{eq:F-size-KK}) this proves~(\ref{eq:KK-goal}) and so we are done.
\end{proof}

\subsection{Traces upper bound}\label{subsec:UB-traces}

In this subsection we complete the proof of Theorem~\ref{t12} by proving the upper bound
$\trace(n,\,n^r,\,\a n) \le O(n^{\C})$, with $\C$ as in~(\ref{eq:C}), for every $r \le \sqrt{n}/\log n$ and $0 \le \a \le 1$.
A proof can be obtained from the proof of Theorem~\ref{theo:sKK-UB} with some effort. For completeness, we give here a self-contained proof.

Put $x=(\C-1)\log n$.
Let $S_1,\ldots,S_\ell$ be $\ell$ independent uniformly random subsets of $[n]$, each of size $x$,
where
\begin{equation}\label{eq:UB-identity}
\frac12\ell := \frac{n^r}{2^x} = \frac{n^{\C}}{(1+\a)^x}.
\end{equation}
%
Let the family $\FF \sub 2^{[n]}$ consist of the union over $j$ of 
all subsets of the set $S_j$; that is, $\FF = \bigcup_{j=1}^\ell 2^{S_j}$.
Let $E_1$ be the event that $|\FF| \ge n^r$, and let $E_2$ be the event that  $\trace(\FF,\,\a n) \le 8n^\C$.
Note that the proof would follow by showing that 
\begin{equation}\label{eq:traces-goal}
\Pr(E_1 \text{ and } E_2) > 0.
\end{equation}
First, we claim that
\begin{equation}\label{eq:F-size}
\Pr(E_1) \ge \frac12 .
\end{equation}
Put $t=3r$.
Denote by $B$ the binomial random variable with parameters $(x,p)$ where $p=x/n$. For every $j \neq j'$ we have
\begin{align*}
\Pr(|S_j \cap S_{j'}| \ge t) &\le 2\Pr(B \ge t) \le 2\binom{x}{t}p^{t} \\
&\le (xp)^{t} = x^{6r}/n^{3r} \le n^{-2r} \le \ell^{-2}, 
\end{align*}
where the first inequality uses Lemma~\ref{lemma:HG} together with the fact that
\begin{equation}\label{eq:UB-x-sqrtn}
x \le r\log n \le \sqrt{n} 
\end{equation}
by the assumed upper bound on $r$, 
and the last inequality uses~(\ref{eq:UB-identity}).
%
Conditioned on the above we have, by taking the union bound over all $\binom{\ell}{2}$ pairs of sets, that
$$|\FF| \ge \ell \bigg(2^x-\binom{x}{\le t}\bigg) \ge \ell \cdot \frac12 2^x = n^r ,$$
where the second inequality uses $t = 3r \le x/2$ (for all $n$ large enough), and the equality uses~(\ref{eq:UB-identity}). This proves~(\ref{eq:F-size}).



We now show that $\Pr(E_2) \ge 1/2$ (that is, $\trace(\FF,\,\a n) \le 8n^\C$ except with probability smaller than $1/2$), thus proving~(\ref{eq:traces-goal}).
Fix $I \sub [n]$ of size $q=\a n$.
Note that 
$$\FF_I = \bigcup_{j=1}^\ell \{ S \cap I \,\vert\, S \sub S_j \} = \bigcup_{j=1}^\ell 2^{S_j \cap I}.$$
For each $1 \le j \le \ell$ consider the random variable $X_j = |2^{S_j \cap I}|$, 
let $X=\sum_{j=1}^\ell X_j$ and note that $|\FF_I| \le X$.
We have
%
\begin{align*}
\Ex(X_j) &= \sum_{h=0}^x 2^h\Pr[|S_j \cap I|=h]
\le 2\sum_{h=0}^x 2^h \binom{x}{h} \a^h (1-\a)^{x-h}\\
&= 2(1-\a)^x\sum_{h=0}^x\binom{x}{h} \Big(\frac{2\a}{1-\a}\Big)^h
= 2(1-\a)^x \Big(1 + \frac{2\a}{1-\a}\Big)^x
= 2(1+\a)^x,
\end{align*}
where the first inequality uses Lemma~\ref{lemma:HG} together with~(\ref{eq:UB-x-sqrtn}).
Thus, by linearity of expectation and by~(\ref{eq:UB-identity}),
\begin{equation}\label{eq:Expectation}
\Ex(X) \le \ell \cdot 2(1+\a)^x  = 4n^{\C} .
\end{equation}
Note that, by our choice of $x$ at the beginning of the proof,
\begin{equation}\label{eq:UB-traces-normalized}
n^\C/2^x = n^\C/2^{(\C-1)\log n} = n.
\end{equation}
Since $X_j \le 2^x$ for every $1 \le j \le \ell$, 
we have that $X/2^x$ is a sum of mutually independent random variables each in $[0,1]$.
We thus apply Claim~\ref{claim:Chernoff} on $X/2^x$, using the fact that $\Ex(X/2^x) \le 4n$ by~(\ref{eq:Expectation}) and~(\ref{eq:UB-traces-normalized}), to deduce that
$$\Pr\big[X \ge 8 n^\C \big] = \Pr\big[X/2^x \ge 8n \big] \le \exp\big(-(4/3)n\big) < 2^{-n}.$$
Using the union bound over all $\binom{n}{\a n} \le \frac12 2^n$ choices of $I \sub [n]$ with $|I|=\a n$ we deduce 
that, except with probability smaller than $1/2$, for every $I \sub [n]$ with $|I|=\a n$ it holds that $|\FF_I| \le 8n^\C$. As mentioned before, together with~(\ref{eq:F-size-KK}) this proves~(\ref{eq:KK-goal}) and so we are done.



\section{Applications}\label{sec:applications}

In this section we give two easy applications of our results on traces, in geometry and in graph theory.

\subsection{Separating halfspaces using few points}
We recall the necessary definitions and the statement of this application.
Let $\HH$ be a family of halfspaces in $\R^d$, and let $P$ be a set of points in $\R^d$. We say that $P$ \emph{separates} $\HH$ if for every pair of distinct halfspaces $H_1 \neq H_2 \in \HH$ there is a point in $P$ that lies in one and outside the other. 
Given $P$ and $\HH$ such that $P$ separates $\HH$, it is interesting to ask how few points in $P$ can we choose while still separating many of the hyperplanes in $\HH$.

\begin{prop}\label{pr:tr}
	Let $P \subset \R^d$ be a set of $n$ points and let $\HH$ be a family of $n^r$ halfspaces in $\R^d$, with $1 \le r \le n^\d$, such that $P$ separates $\HH$.  
	Then there exists a subset $P' \subseteq P$ of at most $n^{1-\d}$ points and
	a subset $\HH' \subseteq \HH$ of at least $n^{\frac{r+1}{2}(1-O(\d))}$ halfspaces such that $P'$ separates $\HH'$.
%
\end{prop}
\begin {proof}
Let $\FF$ be the hypergraph on $P$ with edge set $\{H \cap P \,\vert\, H \in \HH\}$. 
By assumption, for any pair of distinct halfspaces $H_1 \neq H_2 \in \mathcal H$ there is a point $p \in P$ such that $p\in H_1$ and $p\not\in H_2$, or $p \in H_2$ and $p \notin H_1$. In particular, $H_1\cap P \ne H_2 \cap P$, and therefore $|\FF| = |\HH| = n^r$. 
Applying Theorem~\ref{th:main} with $r,\a^{-1} = n^{\d}$, there exists a subset $P' \subseteq P$ of size $n^{1-\d}$ such that $|\FF_{P'}| \ge n^{\C(1-O(\d))} \ge n^{\frac{r+1}{2}(1-O(\d))}$. 
Note that $\FF_{P'}=\{H \cap P' \,\vert\, H \in \HH\}$. Let $\HH' \subseteq \HH$ be obtained by assigning to each member $Q$ of $\FF_{P'}$ an arbitrary halfspace $H' \in \HH'$ with $H' \cap P' = Q$. 
By construction, $|\HH'|=|\FF_{P'}|$, hence it remains to show that $P'$ separates $\HH'$.
Again by construction, for every pair of distinct halfspaces $H_1 \neq H_2 \in \HH'$ we have $H_1 \cap P' \neq H_2 \cap P'$. 
This means that there exists a point $p' \in P'$ such that either $p' \in H_1$ and $p' \notin H_2$, or $p' \in H_2$ and $p' \notin H_1$, thus completing the proof. 
%
\end {proof}

In fact, halfspaces can be replaced in Proposition~\ref{pr:tr} by any family of subsets of $\R^d$, as long as the set of points $P$ separates them.
Indeed, in the proof of Proposition~\ref{pr:tr} we considered for each halfspace only the subset of points in $P$ that are contained in it. The condition that $P$ separates $\mathcal H$ implies that all the corresponding subsets of $P$ are distinct, and that is all that is really needed.
However, it does seem reasonable to expect that for ``favorable'' geometric objects, such as halfspaces, better bounds hold---as we believe is the case.

\begin{remark}
	It is worth noting that the number of halfspaces in 
Proposition~\ref{pr:tr} is necessarily polynomial 
if the dimension $d$ is constant; in fact, an exact formula is known~\cite{Ha}, 
which in particular implies that the number of halfspaces in $\R^d$ 
separated by $n$ points is at most $O(n^d)$. As mentioned above, 
the assertion in Proposition~\ref{pr:tr} holds also when we replace 
halfspaces by any family of $n^r$ subsets, for example, general 
convex sets. 
\end{remark}

%

\subsection{Retaining independent sets in induced subgraphs}\label{subsec:app-graphs}

An \emph{independent set} in a graph is a vertex subset that spans no edges.
Given a graph $G$, one can ask how many independent sets are retained in small subgraphs of $G$. 
Using our result for traces, we easily obtain that if $G$ has $2^{n^{o(1)}}$ independent sets then it must have a subset of at most $n^{1-\d}$ vertices, with $\d$ a sufficiently small constant, retaining asymptotically more than square root of the number of independent sets.

\begin{prop}\label{prop:app-graphs}
	Let $G = (V,E)$ be an $n$-vertex graph, and assume that the number of independent sets in $G$ is $n^r$ with $1 \le r \le n^\d$.
	Then there exists a subset $V' \subseteq V$ of at most $n^{1-\d}$ vertices such that the number of independent sets in the induced subgraph $G[V']$ is at least $n^{\frac{r+1}{2}(1-O(\d))}$.
\end{prop}
\begin{proof}
	Let $\FF$ be the hypergraph on $V$ whose edges are the independent sets in $G$.
	Apply Theorem~\ref{th:main} with $r,\a^{-1} = n^\d$ to deduce that there is a subset $S$ of $V(\FF)=V$ with $n^{1-\d}$ vertices such that $|\FF_S|$, the number of projections of the edges of $\FF$ onto $S$, is at least 
	$n^{\mu(1-O(\d))} \ge n^{\frac{r+1}{2}(1-O(\d))}$.
	The proof follows by observing that, by definition and construction, 
	\begin{align*}
	\FF_S &=\{I \cap S \,\colon\, I \in \FF \} \\ 
	&=\{ I \cap S \,\colon\, I \text{ is an independent set of } G \}
	= \{ I \,\colon\, I \text{ is an independent set of } G[S] \} .
	\end{align*}
\end{proof}

We note that, since Theorem~\ref{th:main} is such an abstract statement, 
we could have replaced the notion of independent sets in a graph 
by any other monotone family (i.e., 
a family closed under taking subsets) of vertex subsets. 
Thus, for example, a statement analogous to 
Proposition~\ref{prop:app-graphs} holds for independent sets in 
hypergraphs, for subsets inducing a subgraph not containing some 
fixed graph $H$, and more.


%

\subsection*{Acknowledgement} We would like to thank Michael Langberg for useful discussions.

\appendix

\section{Proof of Auxiliary Lemmas}\label{sec:aux}

For the proof of Theorem~\ref{th:skk} we use the following bound.

\begin{claim}\label{claim:binom-ratio}
	Let $x,y>0$ and $k,i,\Delta \in \N$ with $\D \le i \le k$.
	If $\binom{y}{k} \le \binom{x}{k-\D}$ then $\frac{\binom{y}{i}}{\binom{y}{k}} \ge i^{-\D} \frac{\binom{x}{i-\D}}{\binom{x}{k-\D}}$.
\end{claim}
\begin{proof}
	First, note that
	\begin{equation}\label{eq:binomial-frac-bound}
	\frac{\binom{z}{a}}{\binom{z}{b}} = 
	\prod_{j=a+1}^{b} \frac{j}{x-j+1} 
	\le b^{b-a}
	\end{equation}
	for all $a \le b \le z$ with $a,b \in \N$. 
	Second, observe that $\frac{j}{x-j+1}$ is:
	\begin{enumerate}
		\item monotone decreasing in $x$,
		\item monotone increasing in $j$.
	\end{enumerate}
	Now, if $y \le x$ then $\frac{\binom{y}{i}}{\binom{y}{k}} \ge \frac{\binom{y}{i-\D}}{\binom{y}{k-\D}} \ge \frac{\binom{x}{i-\D}}{\binom{x}{k-\D}}$ using~(\ref{eq:binomial-frac-bound}) together with items~1 and~2, respectively,
	whereas if $y \ge x$ then
	$$\frac{\binom{y}{i}}{\binom{y}{k}} 
	\ge \frac{\binom{x}{i}}{\binom{x}{k-\D}} 
	= \frac{\binom{x}{i}}{\binom{x}{i-\D}} \cdot \frac{\binom{x}{i-\D}}{\binom{x}{k-\D}}
	\ge \frac{1}{i^{\D}}\cdot \frac{\binom{x}{i-\D}}{\binom{x}{k-\D}},$$
	where the first inequality follows from the statement's assumption,
	and the last inequality uses~(\ref{eq:binomial-frac-bound}).
\end{proof}

%
%

For the proof of Lemma~\ref{lemma:113} we use the following claim, which follows easily from Newton's generalized binomial theorem.
\begin{claim}\label{claim:Newton}
	For every real $x > 0$ we have $2^{x-1} < \sum_{i=0}^{\floor{x}} \binom{x}{i} \le 2^x$.
\end{claim}
\begin{proof}
	We bound the summation in the statement from above and from below as follows;
	$$2^{x-1} < 2^{\floor{x}} \le \sum_{i=0}^{\floor{x}} \binom{\floor{x}}{i} 
	\le \sum_{i=0}^{\floor{x}} \binom{x}{i} 
	\le \sum_{i=0}^\infty \binom{x}{i} = (1+1)^x = 2^x ,$$
	where the last inequality follows from
	\begin{align*}
	\sum_{i=0}^\infty \binom{x}{i} - \sum_{i=0}^{\floor{x}} \binom{x}{i} 
	&= \sum_{j} \bigg(\binom{x}{j}+\binom{x}{j+1}\bigg) 
	= \sum_{j} \binom{x+1}{j+1}\\
	&= \sum_{j} \bigg(\frac{(x+1)x\cdots(x-\floor{x})}{(j+1)!} \prod_{k=\floor{x}+1}^{j-1} (x-k) \bigg) \ge 0
	\end{align*}
	with $j$ running over the values $\floor{x}+1,\floor{x}+3,\ldots$, 
	and where the first equality follows from Newton's generalized binomial formula.
\end{proof}

\begin{proof}[Proof of Lemma~\ref{lemma:113}]
	For real $z \ge 0$ we henceforth abbreviate $\binom{x}{\le z} = \sum_{i=0}^{\floor{z}} \binom{x}{i}$.
	If $k \ge \lceil \frac x 2 \rceil$ we are done since
	\begin{align*}
	2\sum_{i=0}^k \binom{x}{i}\g^i &\ge 2\sum_{i=0}^{\floor{\frac{x}{2}}} \binom{\floor{x}}{i}\g^i
	= \sum_{i=0}^{\floor{\frac{x}{2}}} \bigg(\binom{\floor{x}}{i}+\binom{\floor{x}}{\floor{x}-i}\bigg)\g^i
	\ge \sum_{i=0}^{\floor{\frac{x}{2}}} \bigg( \binom{\floor{x}}{i}\g^i+\binom{\floor{x}}{\floor{x}-i}\g^{\floor{x}-i} \bigg)\\
	&\ge \sum_{i=0}^{\floor{x}} \binom{\floor{x}}{i}\g^i = (1+\g)^{\floor{x}}
	\ge \frac12 (1+\g)^x \ge \frac12\binom{x}{\le x}^{\log(1+\g)}  \ge \frac12\binom{x}{\le k}^{\log(1+\g)} ,
	\end{align*}
	where the penultimate inequality follows from Claim~\ref{claim:Newton},
	and dividing over by $2$ gives the desired result.
	
	We thus assume $k \le \ceil{\frac{x}{2}}-1$ for the remainder of the proof.
	Put
	$$y=\log \binom{x}{\le k} \quad\text{ and }\quad K=\floor{y}+1.$$
	We have $0\le y \le x-1$, which follows using the fact that for $i \le \floor{\frac{x}{2}}$ we have $\binom{x}{i} \le \binom{x}{\floor{x}-i}$ as follows;
	\begin{align*}
	1\le 2\binom{x}{\le k} \le 2\binom{x}{\le\ceil{\frac{x}{2}}-1} \le \sum_{i=0}^{\ceil{\frac{x}{2}}-1} \bigg( \binom{x}{i} + \binom{x}{\floor{x}-i} \bigg) \le
	\binom{x}{\le x} \le 2^x,
	\end{align*}
	where the first inequality follows since $1 \le x$, the second inequality uses our assumption $k \le \ceil{\frac{x}{2}}-1$, and the last inequality uses Claim~\ref{claim:Newton}. 
	Furthermore, we have $K \ge k$, since otherwise $k \ge K+1$ and thus
	$$\binom{x}{\le k} \ge \binom{x}{\le K+1} \ge \binom{y+1}{\le K+1} = \binom{y+1}{\le \floor{y}+2} \ge \binom{y+1}{\le \floor{y}+1} > 2^y,$$
	a contradiction, where the second and third inequalities used the fact that for a positive number $t$, $\binom {t} {\lfloor t \rfloor + 1}$ is non-negative (and $K = \lfloor y \rfloor + 1$), and the last inequality uses Claim~\ref{claim:Newton}. 
	
	Let $a_0,\ldots,a_k$ and $b_0,\ldots,b_{K}$ be given by 
	$$a_i = \binom{x}{i}\text{ for $0 \le i \le k$,}\qquad b_i = \binom{y}{i} \text{ for $0 \le i \le K-1$, and $b_K=\sum_{i=0}^k a_i - \sum_{i=0}^{K-1} b_i$}.$$
	We have that $b_K \ge 0$ since
	$$\sum_{i=0}^{K-1} b_i = \binom{y}{\le y} \le 2^y = \binom{x}{\le k} = \sum_{i=0}^k a_i,$$
	where the inequality follows from Claim~\ref{claim:Newton}.
	Summarizing the properties of the sequences $(a_i)_{i=0}^k$ and $(b_i)_{i=0}^K$, we have $\sum_{i=0}^k a_i = \sum_{i=0}^{K} b_i$, and for every $0 \le i \le k-1$ we have $a_i \ge b_i$ (recall $x > y \ge 0$). 
	Denote $\D_i = a_i-b_i$ $(\ge 0)$ for every $0 \le i \le k-1$.
	We have 
	\begin{align*}
	\sum_{i=0}^{\floor y} \binom{y}{i}\g^i &\le \sum_{i=0}^K b_i \g^i = \sum_{i=0}^{k-1} (a_i-\D_i)\g^i + \sum_{i=k}^K b_i\g^i\\ 
	&\le \sum_{i=0}^{k-1} (a_i-\D_i)\g^i + \Big(\sum_{i=k}^K b_i\Big)\g^k
	= \sum_{i=0}^{k-1} (a_i-\D_i)\g^i + \Big(\sum_{i=0}^{k-1} \D_i + a_k\Big)\g^k\\
	&\le \sum_{i=0}^{k-1} a_i\g^i -\sum_{i=0}^{k-1}\D_i\g^k + \Big(\sum_{i=0}^{k-1} \D_i + a_k\Big)\g^k
	= \sum_{i=0}^k a_i \g^i = \sum_{i=0}^k \binom{x}{i} \g^i.
	\end{align*}
	We thus showed that
	$$\sum_{i=0}^k \binom{x}{i}\g^i \ge 
\sum_{i=0}^{\floor y} \binom{y}{i} \g^i \geq
	\sum_{i=0}^{\floor{y}} \binom{\floor{y}}{i} \g^i = (1+\g)^{\floor{y}}
	\geq  \frac12(1+\g)^{y} = \frac12\binom{x}{\le k}^{\log(1+\g)},$$
	which completes the proof.
\end{proof}


\begin{thebibliography}{99}
	\bibitem{Al}
	N. Alon, On the density of sets of vectors, {\em Discrete Math.} 46
	(1983), 199-202.
	\bibitem{AS}
	N. Alon and J. H. Spencer, {\em The Probabilistic Method, 
		Fourth Edition}, Wiley, 2016, xiv+375 pp.
	\bibitem{An}
	D. Angluin, Computational learning theory: survey and selected bibliography, 
	{\it Proc. 24th Annual ACM Symposium on Theory of Computing} 1992.
	\bibitem{BKK}
	J. Bourgain, J. Kahn and G. Kalai,
	Influential coalitions for Boolean Functions,
	arXiv 1409.3033.
	\bibitem{BR}
	B. Bollob\'as and A. J. Radcliffe,
	Defect Sauer results,
	{\it J. Combinatorial Theory Ser. A} 72 (1995), 189-208.
	\bibitem{Bo}
	J. A. Bondy, Induced subsets, {\it J. Combin. Theory Ser. B} 12
	(1972), 201-202.
	\bibitem{BG}
	B. Bukh and X. Goaoc,
	Shatter functions with polynomial growth rates,
	arXiv 1701.06632.
	\bibitem{CGN}
	O. Cheong, X. Goaoc, and C. Nicaud, Set systems and families of permutations with small traces, {\it European Journal of Combinatorics 34} (2013), 229-239.
	\bibitem{Fr}
	P. Frankl, On the trace of finite sets, {\it J. Combin. Theory Ser.
		A} 34 (1983), 41-45.
\bibitem{Ha}
E. F.  Harding, 
The number of partitions of a set of $N$ points in $k$ dimensions 
induced by hyperplanes, 
{\it Proc. Edinburgh Math. Soc.} (2) 15 1966/1967, 285--289. 

	\bibitem{Ka}
	G. O. H. Katona, A theorem on finite sets, in: {\em Theory of Graphs}
	(Erd\H{o}s, P. and Katona, G. O. H., eds.), Akad\'emiai Kiad\'o,
	Budapest.
	\bibitem{KKL}
	J. Kahn, G. Kalai and N. Linial , The influence of variables on
	Boolean functions, {\it Proc. 29th Annual Symposium on Foundations
		of Computer Science}, 1988.
	\bibitem{Kr} J. B. Kruskal, The number of simplices in a complex,
	in: {\em Mathematical Optimization Techniques}, Univ. California
	Press, Berkeley (1963), 251-278.
	\bibitem{Lovasz} L. Lov\'asz, Combinatorial Problems and Exercises,
	13.31, North-Holland, Amsterdam, 1979.
	\bibitem{Mat}
	J.\ Matou\v sek, Geometric Set Systems, {\it European Congress of
		Mathematics}, Springer (1998) 1–27.
\bibitem{Sa}
	N. Sauer, On the density of families of sets, 
	{\it J. Combin. Theory Ser. A} 13 (1972), 145-147.
	\bibitem{Sh}
	S. Shelah, A combinatorial problem; stability and order for
	models and theories in infinitary languages, {\it Pacific or. Math.} 41 (1972), 271- 276.
	\bibitem{VC}
	V. N. Vapnik and A. Ya. Chervonenkis,
	On the uniform convergence of relative frequencies of events to their probabilities, {\it Theory Probab. Appl.} 16 (1971), 264-280.
\end{thebibliography}
\end{document}